\documentclass[12pt,twoside]{extarticle}

\usepackage[utf8]{inputenc}
\usepackage[english]{babel}
\usepackage{geometry}
\usepackage{fancyhdr}
\usepackage{titlesec}
\usepackage{amssymb}\usepackage{amsthm}\usepackage{amsmath}
\usepackage[]{enumitem}
\usepackage{comment}
\usepackage[dvipsnames]{xcolor}
\usepackage{microtype}
\usepackage{tikz}
\usepackage{hyperref}
\usepackage{mathabx}\usepackage{mlmodern} 

\hypersetup{
    colorlinks,
    linkcolor={red},
    citecolor={blue}
}

\setenumerate{itemsep=0pt,topsep=4pt,parsep=2pt,partopsep=5pt}
\setitemize{itemsep=0pt,topsep=4pt,parsep=2pt,partopsep=5pt}
\renewcommand{\thesection}{\arabic{section}}

\titleformat{\chapter}
    {\Huge \bfseries \centering}
   {}
   {5pt}
    {}
    []
\titleformat{\section}
    {\large \bfseries \centering} 
    {\thesection.} 
    {5pt} 
    {} 
\titleformat{\subsection}
    {\bfseries}
    {\thesubsection.}
    {5pt}
    {}
\titleformat{\subsubsection}
    [runin]
    {\bfseries}
    {\thesubsubsection.}
    {5pt}
    {}

    \DeclareMathAlphabet{\pzocal}{OMS}{zplm}{m}{n}
    \DeclareMathAlphabet{\mathcal}{OMS}{cmsy}{m}{n}
    \def\mathbi#1{\textbf{\em #1}} 
  
\makeatletter
\def\thanks#1{\protected@xdef\@thanks{\@thanks
        \protect\footnotetext{#1}}}
\makeatother

\newcommand{\Rn}{\ensuremath{{\mathbf{R}^{n}}}}
\newcommand{\R}{\ensuremath{{\mathbf{R}}}}

\newcommand{\Rpo}{\ensuremath{{\mathbf{R}_{\geq 0}}}}
\renewcommand{\l}{{\frak{L}_F}}
\newcommand{\ldf}{{\frak{L}_{\mathcal{D}(F)}}}

\renewcommand{\S}{{\mathbf{S}^{n-1}}}
\newcommand{\Integrand}{\mathbf{I}}
\newcommand{\graph}{\text{\normalfont Graph}}
\newcommand{\hypo}{\text{\normalfont Hypo}}
\newcommand{\Extr}{\text{\normalfont Extr}}
\newcommand{\Geo}{\ensuremath{F}\text{\normalfont -Geo}}
\newcommand{\w}{\mathcal{W}}
\renewcommand{\i}{\mathcal{I}}
\renewcommand{\a}{\mathcal{A}}

\renewcommand{\d}{\mathcal{D}}
\newcommand{\Star}{\frak{P}}

\newcommand{\Hp}{{\pzocal{H}^{n-1}}}

\renewcommand{\Tilde}{\widetilde}
\renewcommand{\hat}{\widehat}

\newcommand{\ac}{\text{\normalfont AC}}
\newcommand{\Ort}{\text{\normalfont Ort}}

\newcommand{\norm}[1]{\lVert #1 \rVert}

\newcommand{\spn}{\ensuremath{\text{\normalfont span}^+}}
\newcommand{\Cont}{\text{\normalfont Cont}}

\renewcommand{\theta}{\vartheta}
\renewcommand{\phi}{\varphi}

\renewcommand{\epsilon}{\varepsilon}

\newtheoremstyle{defbfnote}
    {15pt} 
    {15pt} 
    {} 
    {} 
    {\bfseries} 
    {.} 
    {.5 em} 
    {\thmname{#1}\thmnumber{ #2}\thmnote{ (#3)}}
\newtheoremstyle{bfnote}
    {15pt} 
    {15pt} 
    {\itshape} 
    {} 
    {\bfseries} 
    {.} 
    {.5 em} 
    {\thmname{#1}\thmnumber{ #2}\thmnote{ (#3)}}
\newtheoremstyle{scesempio}
    {10pt} 
    {10pt} 
    {} 
    {} 
    {\bfseries} 
    {.} 
    {.5 em} 
    {\thmname{#1}\thmnumber{ #2}\thmnote{ (#3)}}
\newtheoremstyle{bfesercizio}
    {20pt} 
    {10pt} 
    {\itshape} 
    {} 
    {\bfseries} 
    {.} 
    {.5 em} 
    {\thmname{#1}\thmnumber{ #2}\thmnote{ (#3)}}
\newtheoremstyle{scsvolgimento}
    {10pt} 
    {30pt} 
    {} 
    {} 
    {\sc} 
    {.} 
    {.5 em} 
    {\thmname{#1}}
    
\theoremstyle{bfnote}
\newtheorem{theorem}{Theorem}

\newtheorem{corollary}[theorem]{Corollary}
\newtheorem{lemma}[theorem]{Lemma}

\theoremstyle{defbfnote}
\newtheorem*{definition}{Definition}

\theoremstyle{bfesercizio}

\theoremstyle{scsvolgimento}

\theoremstyle{scesempio}
\newtheorem{example}[theorem]{Example}

\theoremstyle{bfnote}

\title{\large \bfseries\uppercase{Complete Characterization of Anisotropic  Geodesics in the Euclidean Space}}
\author{{\sc pietro aldrigo\textsuperscript{\textdagger}} \thanks{\textsuperscript{\textdagger}{\sc Universit\"at Bern, Mathematisches Institut (MAI), Sidlerstrasse 12, 3012 Bern, Schweiz}. Email address: \texttt{pietro.aldrigo@unibe.ch}}
\thanks{The author is supported by the Swiss National Science Foundation, grant number 200021-228012. }
}

\date{}

\begin{document}

\maketitle
\begin{abstract}
    Let $F$ be a lower semicontinuous, 1-homogeneous positive function defined on $\Rn$. We provide a characterization of absolutely continuous paths that minimize the anisotropic $F$-length between two points. The characterization is achieved by establishing a connection between the minimizing paths and the geometry of the anisotropic $F$-isoperimetric set.
\end{abstract}


\section{Introduction and main results}\label{intro}

The study of geodesics has a rich history in several areas of mathematics (see e.g. \cite{busemann2005geometry,fialkow1939conformal,karney2013algorithms,leonardi2008end,montgomery2002tour}) and its applications range from path planning in robotics \cite{wu2016path,4376799} to image processing \cite{demonceaux2011central}, and more. On the other hand, significant advances have been made in the study of the geometric properties of sets arising as critical points of anisotropic functionals (e.g. \cite{de2018rectifiability,de2020existence,de2018minimization,de2020uniqueness,franceschi2023isoperimetric,leonardi2005isoperimetric,Monti+2008+93+121,taylor1974existence,taylor1975unique,taylor1978crystalline,wulff1901frage}). The present work lies at the intersection of the two aforementioned fields. We present a complete characterization of anisotropic geodesics (definition given below) in Euclidean space, achieved through the establishment and application of a connection between these geodesics and the geometric properties of anisotropic $F$-isoperimetric set.

\subsection{The $\mathbi{F}$-geodesic problem}
Throughout this work $n$ is an integer greater or equal than 2. We denote by $\S\subseteq \Rn$ the $(n-1)$-dimensional unit sphere centered at the origin. 

A function $F:\Rn\to\R$ is 1-homogeneous if $F(\lambda x)= \lambda F(x)$ for every $\lambda\geq 0$ and $x\in \Rn$. Observe that any 1-homogeneous function is univocally  determined by its values on $\S$. We say that a 1-homogeneous function is positive if it is positive in $\S$. An \emph{integrand} is a lower semicontinuous, 1-homogeneous positive function and the set of all integrands is denoted by $\Integrand$.

Denote by $\ac([0,1];\Rn)$ the family of absolutely continuous functions $\gamma:[0,1]\to\Rn$. It is well-known that if $\gamma\in \ac([0,1];\Rn)$ then $\gamma$ admits a derivative $\dot\gamma(t)$ at almost every $t\in (0,1)$ and  $t\mapsto \dot\gamma(t)$ belongs to $L^1([0,1])$. We say that $\gamma$ is \emph{regular} if $|\dot\gamma(t)|>0$ for almost every $t\in (0,1)$.

\begin{definition}[Anisotropic $\mathbi{F}$-length]
    Let $\gamma\in \ac([0,1];\Rn)$ and let $F$ be an integrand. The  anisotropic $F$-length (or $F$-length) of $\gamma$  is the quantity 
    \begin{align*}
        \l(\gamma):= \int_0^1 F(\dot\gamma(t))\,dt.
    \end{align*}
\end{definition}

Observe that the definition of $F$-length is invariant under reparametrization of $\gamma$, i.e. if $\gamma,\rho\in \ac([0,1];\Rn)$ and $\rho(t)=\gamma(\tau(t))$ for some strictly increasing function $\tau:[0,1]\to[0,1]$ such that $\tau(0)=0$ and $\tau(1)=1$, then 
\begin{align*}
    \l(\rho) = \int_0^1 F(\dot\gamma(\tau(t)))\tau'(t)\,dt = \int_0^1 F(\dot\gamma(s))\,ds = \l(\gamma).
\end{align*}
Moreover, in the special case of $F\big|_\S\equiv 1$, the $F$-length of a curve $\gamma\in \ac([0,1];\Rn)$ coincides with the classical length of $\gamma$.

The $F$-geodesic problem associated to $(x,y)\in \Rn\times\Rn$ is the following:
\begin{align}\label{GP}\tag{GP}
    \text{minimize} \quad \l(\gamma)\quad\text{over}\quad \gamma\in \ac([0,1];\Rn)\text{ s.t. }\gamma(0) = x \text{ and }\gamma(1)=y.
\end{align}
We call the solutions (if any) of the problem \eqref{GP} \emph{$F$-geodesics from $x$ to $y$}, and we collect all such solutions into the set $\Geo(x,y)$. We say that $\gamma\in \ac([0,1];\Rn)$ is a \emph{$F$-geodesic} and write $\gamma\in \Geo$ if $\gamma$ is an $F$-geodesic from $\gamma(0)$ to $\gamma(1)$.

\subsection{Main results}

Let $F\in \Integrand$ be a fixed integrand. For each $v\in \S$ consider the half-space 
\begin{align*}
    H_v:=\{x\in \Rn: \langle x,v\rangle \leq F(v)\}.
\end{align*}
The \emph{$F$-crystal} is the convex set
    \begin{align*}
        K_F := \bigcap_{v\in \S} H_v.
    \end{align*}
It turns out that, under the standing assumptions of $F$, $K_F$ is a compact set containing $0$ in its interior. This set is also known in the literature as Wulff's set and it enjoys the following anisotropic isoperimetric property. Let $\Omega\subseteq \Rn$ be a set of finite perimeter (see e.g. \cite[Chapter 5]{EvGar} for the definition and main properties of these sets). The \emph{$F$-perimeter of $\Omega$} is defined as
\begin{align*}
    \text{Per}_F(\partial\Omega):= \int_{\partial \Omega} F(\nu_\Omega(x))\,d\Hp(x),
\end{align*}
where $\partial \Omega$ is the (reduced) boundary of $\Omega$, $\nu_\Omega(x)$ is the outer unit normal to $\partial \Omega$ at $x$ and $\Hp$ denotes the $(n-1)$-dimensional Hausdorff measure. Denoting by $|\cdot|$ the Lebesgue measure in $\Rn$  and setting $\omega_n:=|\{x\in \Rn : |x|\leq 1\}|$, it turns out that 
\begin{align}\label{isop}
    \frac{\text{Per}_F(\partial\Omega)}{|\Omega|^{\frac{n-1}{n}}} \geq \frac{\text{Per}_F(\partial K_F)}{|K_F|^\frac{n-1}{n}} = n \omega_n^\frac{1}{n}
\end{align}
holds for every admissible $\Omega\subseteq \Rn$. Moreover, equality in \eqref{isop} holds if and only if, up to sets of measure zero, $\Omega$ is homothetic to $K_F$ (see e.g. \cite{esposito2005quantitative,taylor1974existence,taylor1975unique,taylor1978crystalline,wulff1901frage}). 

For any subset $\Omega\subseteq \Rn$, the \emph{polar body of $\Omega$} is defined as
\begin{align*}
    \Star \Omega:=\left\{z\in \Rn : \langle z,x \rangle \leq 1 \,\forall x\in \Omega\right\}.
\end{align*}
Since $\Star\Omega$ is the result of the intersections of the half-spaces $\left\{z\in \Rn : \langle z,x\rangle \leq 1\right\}$ for any $x\in \Omega$, then $\Star \Omega$ is a convex subset containing the origin. Moreover, if $\Omega$ is a convex subset containing $0$, then $\Star\Star \Omega = \Omega$ (see Lemma \ref{Lem_starstar}).

If $\Omega\subseteq \Rn$ and $\alpha\geq 0$, the set $\alpha \Omega$ is the set containing all of the elements $\alpha x$ for each $x\in \Omega$.
We define the function $\norm{\cdot}_F:\Rn\to \Rpo$ as 
\begin{align*}
    \norm{z}_F := \min \{\lambda \geq 0 : x\in \lambda \Star K_F\}.
\end{align*}
Notice that $\norm{\cdot}_F$ may fail to be a norm only because, in general, $\norm{-z}_F\neq \norm{z}_F$. For any $z\in \Rn\backslash\{0\}$, $z/\norm{z}_F$ belongs to $\partial (\Star K_F)$. Therefore,
\begin{align*}
    \left\langle z,x\right\rangle \leq \norm{z}_F \,\, \forall x\in K_F\quad \text{and}\quad \exists \overline{x}\in K_F \,\,:\,\, \left\langle z,\overline{x}\right\rangle = \norm{z}_F.
\end{align*}

Given $F\in \Integrand$, we define the \emph{convex envelope of $F$} as 1-homogeneous positive the function $\d(F):\Rn \to \Rpo$
\begin{align*}
    \d(F)(x):= \sup_{v\in \S}\left\{\inf_{\substack{w\in \S\\ \langle x,w\rangle >0 }} \left\{F(w)\frac{\langle v, x\rangle}{\langle v,w \rangle}\right\}\right\}.
\end{align*}
Observe that $\d(F)\leq F$ for every $F\in \Integrand$.
The \emph{contract set of $F$} is 
\begin{align*}
    \Cont(F):=\left\{x\in \Rn : F(x) = \mathcal{D}(F)(x)\right\}
\end{align*}
As both $F$ and $\d(F)$ are 1-homogeneous, if $x\in \Cont(F)$ then $\lambda x\in \Cont(F)$ for every $\lambda\geq 0$. In particular, $0\in \Cont(F)$ for every integrand $F$. We say that $F\in \Integrand$ is \emph{convex} if $\Cont(F)=\Rn$.

Let $K\subseteq\Rn$ be a compact subset and fix $v\in \S$. The \emph{supporting hyperplane of $K$ associated with $v$} is the (affine) hyperplane $\pi_v$ such that 
\begin{align*}
    \pi_v = \left\{z\in \Rn : \langle z,v\rangle = \alpha\right\}\quad \text{and}\quad \max_{y\in K}\left\{\langle y,v\rangle\right\} = \alpha.
\end{align*}

The first main result of this work is the following characterization of the $F$-geodesics.

\begin{theorem}\label{th_MainResult}
    Let $F\in\Integrand$ be an integrand and $\gamma\in \ac([0,1];\Rn)$ be regular. Define $\hat{v}:=\gamma(1)-\gamma(0)$ and let $\overline{x}\in K_F$ such that $\norm{\hat{v}}_F = \langle \hat{v}, \overline{x}\rangle$. The following are equivalent:
    \begin{enumerate}[label = (\arabic*)]
        \item $\gamma$ is a $F$-geodesic;
        \item $\l(\gamma)=\norm{\hat{v}}_F$;
        \item for almost every $t\in (0,1),$ $\dot\gamma(t)\in \Cont(F)$ and $\overline{x} + (\dot\gamma(t))^\perp$ is a supporting hyperplane for $K_F$ at $\overline{x}$.
    \end{enumerate}
\end{theorem}

For any $r>0$ and $x\in \Rn$, we define the $F$-geodesic (closed) ball of center $x$ and radius $r$ as the set
\begin{align*}
    \left\{\gamma(1): \gamma \in \Geo,\,\gamma(0)=x, \l(\gamma)\leq r\right\}.
\end{align*}
Then, as an immediate consequence of Theorem \ref{th_MainResult}, we deduce the following relation between the $F$-crystal and the $F$-geodesic balls.

\begin{corollary}
    Let $F\in \Integrand$ be an integrand. Then the $F$-crystal $K_F$ and the $F$-geodesic unitary ball centered at the origin are one the polar body of the other.
\end{corollary}

Using Theorem \ref{th_MainResult} together with some further remarks, we prove that if the integrand $F$ is convex, then line segments are always $F$-geodesics (Corollary \ref{cor:FconvSegments}). However, as demonstrated in Example \ref{ex:Fconv}, it is possible that, even with a convex integrand $F$, line segments are not the sole $F$-geodesics.

Let $K\subseteq \Rn$ be any convex set. For each $y\in \partial K$, we define the \emph{cone of normal directions of $\partial K$ at $y$} as 
\begin{align*}
    \mathcal{N}_y\partial K := \left\{v\in \Rn : \langle N, x-y\rangle \leq 0 \quad \forall x\in K\right\}.
\end{align*}
It turns out that $\mathcal{N}_y\partial K\cap \S$ is a singleton at $\Hp$-almost every point $y\in \partial K$. For any such points, we denote by $N_{\partial K}(y)$ the unique element of $\mathcal{N}_y\partial K\cap \S$ and we call it \emph{(outer) unit normal to $\partial K$ at $y$}. Whenever $N_{\partial K}(y)$ is defined, 
\begin{align*}
    (N_{\partial K}(y))^\perp := \left\{z\in \Rn: \langle z, N_{\partial K}(y)\rangle = 0\right\}
\end{align*}
is the tangent space of $\partial K$ at $y$. The \emph{affine tangent space of $\partial K$ at $y$} is given by $y + (N_{\partial K}(y))^\perp$. By convexity of $K$, it is easy to see that $y+(N_{\partial K}(y))^\perp = \pi_{N_{\partial K}(y)}$ at every $y\in \partial K$ such that $N_{\partial K}(y)$ is defined. The set of \emph{all orthogonal directions to $\partial K$} is 
\begin{align*}
    \Ort(\partial K):= \left\{v\in \Rn\backslash\{0\}:  \exists y\in \partial K \text{ s.t. } \exists N_{\partial K}(y)=\frac{v}{|v|}\right\},
\end{align*}

Let $F\in \Integrand$ and fix two distinct points $x,y\in \Rn$. 
We say that two curves $\gamma,\rho\in\Geo(x,y)$ are equivalent, and we write $\gamma\sim\rho$ if there exists an increasing reparametrization $\tau:[0,1]\to[0,1]$ of $\rho$ such that $\gamma = \rho\circ \tau$. With ${\Geo_{/\sim}}(x,y)$ we indicate the set of equivalence classes of $\Geo(x,y)$ with respect to the equivalence relation $\sim$. The second main result is the following. 

\begin{theorem}\label{Cor_Main}
    Let $F\in \Integrand$ and $x,y\in \Rn$ such that $x\neq y$. Then the $F$-geodesic problem \eqref{GP} admits a solution. More precisely:
    \begin{enumerate}[label = (\arabic*)]
        \item \label{cor_i} if $y-x\in \Ort(\partial K_F)$, then ${\Geo_{/\sim}}(x,y)$  contains one and only element, and representative of it is the line segment $t\mapsto (1-t)x + ty$.
        \item \label{cor_ii} if $y-x\not \in \Ort(\partial K_F)$, then ${\Geo_{/\sim}}(x,y)$ contains infinitely many elements.
    \end{enumerate}
\end{theorem}

\section{Preliminaries}

\begin{definition}[Convex hull]
    Let $\Omega\subseteq \Rn$ be a subset. The convex hull of $\Omega$ is the set
    \begin{align*}
        [\Omega]:=\left\{(1-\lambda)x + \lambda y: x,y\in \Omega\right\}.
    \end{align*}
\end{definition}

Notice that the convex hull of a set $\Omega$ is the “smallest” convex set containing $\Omega$, in the sense that if $U\supseteq \Omega$ is convex, then $U\supseteq [\Omega]$.

\begin{lemma}\label{Lem_starstar}
    For every $\Omega\subseteq \Rn$, $\Star\Star\Omega = [\Omega \cup\{0\}]$. 
\end{lemma}

\begin{proof}
    First we prove that $[K\cup \{0\}]\subseteq \Star\Star \Omega$. By virtue of the above remark, it is enough to show that $K\subseteq \Star\Star \Omega$. Fix $x\in \Omega$. By definition of polar body then
    \begin{align*}
        \langle x,z\rangle\leq 1\quad \forall z\in \Star\Omega.
    \end{align*}
    Therefore $x\in \Star\Star\Omega$.

    For the converse inclusion, suppose the existence of a point $x\in (\Star\Star\Omega)\backslash [\Omega\cup\{0\}]$. Then there exists an affine hyperplane $\pi=\{x\in \Rn : \langle x,v\rangle = \alpha\}$, for some $v\in \S$ and $\alpha\geq0$ separating the sets $\{x\}$ and $[\Omega \cup \{0\}]$, i.e.
    \begin{align}\label{starstar}
        \langle y,v\rangle < \alpha \quad \forall y\in [\Omega\cup \{0\}]\quad \text{and} \quad \langle x,v \rangle > \alpha.
    \end{align}
    Since $0\in[\Omega\cup \{0\}]$, then $\alpha>0$. Moreover, the first of \eqref{starstar} implies $v/\alpha\in \Star \Omega$. This is, combined with the second of \eqref{starstar}, contradicts the fact that $x\in \Star\Star\Omega$.

\end{proof}

Define the operators $\w,\i,\a,\d:\Integrand \to \Integrand$ as 
\begin{gather*}
    \w(F)(v):= \inf_{\substack{w\in \S\\\langle v,w\rangle>0}}\left\{\frac{F(w)}{\langle v, w\rangle}\right\},
    \quad
    \i(F)(v):=\frac{1}{F(v)},
    \\
    \a(F)(v):= \sup_{w\in\S}\left\{F(w)\langle v,w \rangle\right\},
    \quad
    \d(F):=\a\circ \w(F)
\end{gather*}
for all $F\in \Integrand$ and $v\in \S$, and extended by 1-homogeneity in $\Rn$. It is easy to see that the inequalities $\w(F)\leq F$ and $F\leq \a(F)$ hold true for every integrand $F$. Moreover, for any $v\in \S$ and $F\in \Integrand$,
\begin{align*}
    \d(F)(v) :=  \sup_{u\in \S}\left\{\inf_{\substack{w\in \S\\\langle v,w\rangle>0}}\left\{F(w)\frac{\langle u,v\rangle}{\langle u,w\rangle}\right\}\right\} \leq \sup_{u\in \S}\left\{F(v)\right\} = F(v).
\end{align*}
Hence 
\begin{align}\label{operatorInequalities}
    \d(F)\leq F\quad \forall F\in \Integrand.
\end{align}

If $F\in \Integrand$, we call \emph{polar graph of} $F$ and \emph{polar hypograph of} $F$ the sets
\begin{gather*}
    \graph(F):=\{F(v)v\in \Rn: v\in \S\} = \{x\in \Rn : |x|= F(x/|x|)\},\\ \hypo(F):=\{\lambda F(v)v\in \Rn: v\in \S,\,0\leq\lambda\leq 1\} = \{x\in \Rn: |x|\leq F(x/|x|)\}
\end{gather*}
respectively.
Notice that for every integrand $F$, the origin of $\Rn$ belongs to the interior of $\hypo(F)$ and $\graph(F)\subseteq\partial(\hypo(F))$.
Moreover, $K_F= \hypo(\w(F))$. Indeed, both sets contain the origin $0\in \Rn$ and, if $x\in \hypo(\w(F))\backslash\{0\}$, by definition of $\w(F)$ we have 
    \begin{align*}
        |x|\leq \inf_{\substack{w\in\S\\ \langle x/|x|,w\rangle>0}} \left\{\frac{F(w)}{\langle x/|x|, w\rangle}\right\}.
    \end{align*}
    Therefore $x \in H_w$ for every $w\in \S$. This proves the inclusion “$\supseteq$”. To prove the other, fix a point $y\in K_F\backslash\{0\}$. Then 
    \begin{align*}
        |y|\left\langle \frac{y}{|y|},w\right\rangle = \langle y , w \rangle \leq F(w)\quad \forall w\in \S.
    \end{align*}
    Therefore,  
    \begin{align*}
        |y|\leq \frac{F(w)}{\langle y/|y|, w\rangle}\quad \forall w\in \S\text{ s.t. }\langle y/|y|,w\rangle>0.
    \end{align*}
    Hence $y\in \hypo(\w(F))$.

Let $\Omega\subseteq \Rn$ be any bounded subset. The \emph{support function of $\Omega$} is the function $\beta_\Omega:\Rn\to \R$ given by
\begin{align*}
    \beta_\Omega(x):= \sup_{y\in \Omega}\left\{\langle x,y\rangle \right\}
\end{align*}
for every $x\in \Rn$. Notice that the support function of a set is always convex and 1-homogeneous. Moreover, $\beta_\Omega$ is positive (as a 1-homogeneous function) if and only if $0$ is contained in the interior of $\Omega$. A rather trivial, yet important, property of the support function is that the hyperplane $\pi_v:= \beta_\Omega(v)v + v^\perp$ is a supporting hyperplane for $\Omega$ for every $v\in \S$.

\begin{lemma}\label{interiorConvexHull}
    Let $F\in \Integrand$ be an integrand.
    \begin{enumerate}[label = (\roman*)]
        \item \label{keyLem_*}$\d(F)$ is the support function of $K_F$. In particular, $\d(F)$ is a convex function. Moreover, if $G$ is any other convex, 1-homogeneous positive function such that $G\leq F$, then $G\leq \d(F)$.
        \item \label{keyLem_**} For every $v\in \S$ and $\overline{x}\in K_F$, the following are equivalent:
        \begin{enumerate}[label = (\alph*) ]
            \item \label{a}$\d(F)(v)= \langle v,\overline{x}\rangle$;
            \item\label{b} $\overline{x}\in\d(F)(v)v +v^\perp$;
            \item \label{c}$\overline{x} + v^\perp$ is a supporting hyperplane for $K_F$.
        \end{enumerate}
        \item \label{keyLem_iv} $\Ort(\partial K_F)\subseteq \Cont(F)$.
    \end{enumerate}
\end{lemma}

\begin{proof}

    \ref{keyLem_*} Fix $v\in \S$. Then
    \begin{align*}
        \d(F)(v) = \sup_{w\in \S}\left\{\w(F)(w)\langle v,w\rangle \right\} = \sup_{y\in K_F}\left\{\langle v,y\rangle \right\} = \beta_{K_F}(v).
    \end{align*}
    Since both $\d(F)$ and $\beta_{K_F}$ are 1-homogeneous, they coincide in $\Rn$. Suppose now $G$ to be a convex, 1-homogeneous positive function. Then $G$ is the support function of the set
    \begin{align*}
        \Omega_G:=\{y\in\Rn : \langle v,y\rangle \leq G(v)\,\forall v\in \S\}.
    \end{align*}
    Therefore, if $G\leq F$, then $\Omega_G\subseteq K_F$ and, as $\d(F)$ is the support function of $K_F$, then $G\leq \d(F)$.
    
   \ref{keyLem_**} Fix $v\in \S$ and $\overline{x}\in K_F$. The equivalence of \ref{b} and \ref{c} is an immediate consequence of \ref{keyLem_*}. Therefore, it is enough to prove that  \ref{a} holds if and only if \ref{c} holds. 
    
    Suppose that $\overline{x} + v^\perp$ is a supporting hyperplane for $K_F$, then, by virtue of \ref{keyLem_*}, $\d(F)(v)v = \overline{x} + w$, for some $w\in v^\perp$. Therefore, taking the scalar product of both with $v$,
    \begin{align*}
        \langle v,\overline{x}\rangle = \d(F)(v)\langle v,v\rangle = \d(F)(v).
    \end{align*}
    Viceversa, if $\d(F)(v) = \langle v,\overline{x}\rangle$, then, using the definition of support function,
    \begin{align*}
        \langle v,\overline{x}\rangle \geq \langle v,y\rangle\quad \forall y\in K_F.
    \end{align*}
    Thus, as $\overline{x}\in K_F$, the plane $\overline{x} + v^\perp$ is a supporting hyperplane for $K_F$.

    \ref{keyLem_iv} Fix $\overline{x}\in \partial K$ such that the (outer) unit normal $v:=N_{\partial K_F}(\overline{x})$ is well defined. Then $\overline{x} + v^\perp$ is a supporting hyperplane of $K_F$, thus, by \ref{keyLem_**}, 
    \begin{align*}
        \d(F)(v)= \langle\overline{x},v\rangle. 
    \end{align*}
    On the other hand, $K_F$ is the intersection of the halfspaces 
    \begin{align*}
        \{z\in \Rn : \langle z,w\rangle \leq F(w)\}\quad w\in \S
    \end{align*}
    and $F$ is lower semicontinuous. Hence, for every $y\in \partial K_F$ there exists $\overline{w}(y)\in \S$ such that 
    \begin{align*}
        F(\overline{w}(y)) = \langle y,\overline{w}(y)\rangle.
    \end{align*}
    Observe that $F(\overline{w}(y))\overline{w}(y)+ (\overline{w}(y))^\perp$ is a supporting hyperplane passing through $y$. Since we are assuming that $\partial K_F$ admits a tangent space at $\overline{x}$, then $\overline{w}(\overline{x})= v$. This proves that $F(v) = \d(F)(v)$.

\end{proof}

As a consequence of Lemma \ref{interiorConvexHull}\ref{keyLem_*}, an integrand $F\in \Integrand$ is convex in the sense of Section \ref{intro} if and only if the function $x\mapsto F(x)$ is convex in $\Rn$ in the classical sense. 

\section{Proof of Theorem \ref{th_MainResult} and further remarks}

\begin{lemma}\label{O_F=PK_F}
    Let $O_F:=\hypo(\i\circ\d(F))$. Then $\Star O_F = K_F$ and $\Star K_F = O_F$. In particular, $O_F$ is a compact and convex subset containing $0$ in its interior. 
\end{lemma}

\begin{proof}
    By virtue of Lemma \ref{Lem_starstar} and the fact that $K_F$ is a compact convex subset containing $0$ in its interior, it is enough to show that $O_F = \Star K_F$.
    Using the definition of hypograph, $F$-crystal and of $\d = \a \circ \w$, 
    \begin{align*}
        O_F = \left\{x\in \Rn : \d(F)(x)\leq 1\right\}= \left\{x\in \Rn: \langle\w(F)(v)v,x\rangle \leq 1 \,\,\forall v\in \S\right\} = \Star K_F.
    \end{align*} 
\end{proof}

\begin{corollary}\label{Cor:Tan}
    For any $v\in \Rn$, $\d(F)(v) = \norm{v}_F$. 
\end{corollary}

\begin{proof}
        Recall the definition of $\norm{\cdot}_F$ given in the introduction. Then, by Lemma \ref{O_F=PK_F},
\begin{align}\label{norm=D}
    \norm{v}_F = \min\left\{\lambda \geq 0 : v\in \lambda O_F\right\} = \min\{\lambda \geq 0 : \d(F)(v)\leq \lambda\} = \d(F)(v),
\end{align}
for every $v\in \Rn$.
    
\end{proof}

We are finally ready to prove Theorem \ref{th_MainResult}.

\begin{proof}[Proof of Theorem \ref{th_MainResult}]
    Let $F\in \Integrand$ and fix a curve $\gamma\in \ac([0,1];\Rn)$. Set $\hat{v}:=\gamma(1)-\gamma(0)$ and $\overline{x}\in K_F$ such that $\norm{\hat{v}}_F=\langle \hat{v},\overline{x} \rangle$. On the one hand, using \eqref{operatorInequalities} and Jensen's inequality,
\begin{align}\label{pf_inequalities}
    \l(\gamma) = \int_0^1 F(\dot\gamma(t))\,dt \geq \int_0^1 \d(F)(\dot\gamma(t))\,dt \geq \d(F)(\hat{v}),
\end{align}
and the two inequalities are equalities if and only if $\dot\gamma(t)\in \Cont(F)$ for almost every $t\in (0,1)$ and $\d(F)$ is linear in the image of $\dot\gamma$. On the other hand, by Corollary \ref{Cor:Tan}, $\d(F)$ is linear in the image of $\dot\gamma$ if and only if $\d(F)(\dot\gamma(t))= \langle \dot\gamma(t),\overline{x}\rangle$ for almost every $t\in (0,1)$. Thus by virtue of Lemma \ref{interiorConvexHull}\ref{keyLem_**}, the two inequalities of \eqref{pf_inequalities} are equalities if and only if, for almost every $t\in(0,1)$, $\dot\gamma(t)\in \Cont(F)$ and $\overline{x} + (\dot\gamma(t))^\perp$ is a supporting hyperplane for $K_F$ at $\overline{x}$.

\end{proof}

\begin{corollary}\label{cor:FconvSegments}
    Let $F$ be a convex integrand and fix $\gamma\in \ac([0,1];\Rn)$. If $\gamma$ is a reparametrization of a segment then $\gamma$ is a $F$-geodesic.
\end{corollary}

\begin{proof}
    Under the standing assumptions, for every path $\gamma\in \ac([0,1];\Rn)$ of the form $\gamma(t)=x_0+\tau(t)\hat{v}$, $\hat{v}\in \Rn\backslash\{0\}$, both the inequalities in \eqref{pf_inequalities} are equalities. Thus, is a $F$-geodesic.
    
\end{proof}

Now we exhibit a counter-example for the converse of Corollary \ref{cor:FconvSegments}, demonstrating that even when $F$ is a convex integrand, not every $F$-geodesic needs to be a line segment.

\begin{example}\label{ex:Fconv}
\newcommand{\s}{{\mathbb{S}^2}}
    Consider the 1-homogeneous positive function $F:\R^2\to \Rpo$ defined  as  
    \begin{align*}
        F(x,y):= |x| + |y| \quad \forall (x,y)\in \R^2.
    \end{align*}
    Since $F$ is a convex function, by Lemma \ref{interiorConvexHull}\ref{keyLem_*}, $F$ is a convex integrand. Fix $\hat{v}:=(1,1)$ and let $\gamma_0 \in \ac([0,1];\R^2)$ be the function $\gamma_0(t):= t\hat{v}$ for every $t\in [0,1]$. By virtue of Corollary \ref{cor:FconvSegments}, $\gamma_0$ is a $F$-geodesic. Therefore,
    \begin{align*}
       \min_{\substack{\gamma(0) = (0,0)\\\gamma(1)=(1,1)}} \left\{\l(\gamma)\right\}= \l(\gamma_0) = \int_0^1 F(\dot\gamma_0(t))\,dt = 2.
    \end{align*}
\begin{figure}[t]
        \centering
        \begin{tikzpicture}[x=2.5cm,y=2.5cm]
    \draw[very thin,-] (-1.2,0) -- (1.2,0); 
    \draw[very thin,-] (0,-1.2) -- (0,1.2); 

    \foreach \j [evaluate=\j as \aj using 2*pi*\j/(100)] in {0,...,99}
    {
        \coordinate (P) at ({(abs(cos(\aj r))+abs(sin(\aj r)))*cos(\aj r)}, {(abs(cos(\aj r))+abs(sin(\aj r)))*sin(\aj r)});
        
        
        \draw[ultra thin,lightgray]  ({(abs(cos(\aj r))+abs(sin(\aj r)))*cos(\aj r) - 1*sin(\aj r)} , {(abs(cos(\aj r))+abs(sin(\aj r)))*sin(\aj r) + 1*cos(\aj r)}) -- 
            ({(abs(cos(\aj r))+abs(sin(\aj r)))*cos(\aj r) + 1*sin(\aj r)} , {(abs(cos(\aj r))+abs(sin(\aj r)))*sin(\aj r) - 1*cos(\aj r)});
    }
    \draw[domain=0:2*pi,smooth,variable=\t,color=black,thick,samples=200] plot
        ({abs(cos(\t r))+abs(sin(\t r)))*cos(\t r)},
         {(abs(cos(\t r))+abs(sin(\t r)))*sin(\t r)});
         
    \draw[color=blue,thick] (-1,-1) rectangle (1,1);
    
    \draw[semithick, color=red] (0,0) -- (1,1); 
    \draw[semithick, color=red] (0,0) -- (0,1); 
    
    \draw[domain=0:1,smooth,variable=\x,color=orange,semithick] plot(\x, {\x^(1/2)});

    \draw[violet, semithick, ->] (0,1/3) -- ++(0,1/3);
    \draw[violet,semithick, ->] (0,1) -- ++(0,1/3);

    \draw[violet, semithick, ->] (1/3,1/3) -- ++(1/4.5,1/4.5);
    \draw[violet,semithick, ->] (1,1) -- ++(1/4.5,1/4.5);
    
    \draw[brown,semithick, ->] (1/2,{1/(2^(1/2))}) -- ++(1/4,{1/(4*(2^(1/2)))});
    \draw[brown,semithick, ->] (1,1) -- ++(1/4,{1/(4*(2^(1/2)))});

    \draw[violet, semithick] (-1.1,1.01) -- (1.1,1.01);
    \draw[violet, semithick] (0.3,1.7) -- (1.7,0.3);
    \draw[brown, semithick] (0.45,1.77) -- (1.58,0.179);
        \end{tikzpicture}
        \caption{graphical representation of Example \ref{ex:Fconv}.}
        \label{fig:1}
    \end{figure}
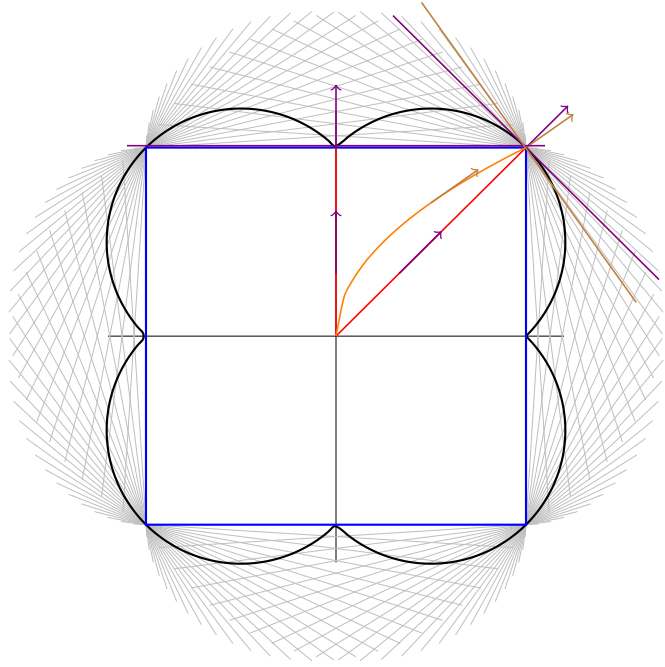
    Let $f,g:[0,1]\to[0,1]$ be any two absolutely continuous, strictly, increasing bijective functions and let $\gamma(t):= (f(t),g(t))$ for every $t\in [0,1]$. Then $\gamma\in \ac([0,1];\R^2)$ and satisfies $\gamma(0)=(0,0)$ and $\gamma(1)=(1,1)$. Moreover, by the fundamental theorem of calculus
    \begin{align*}
        \l(\gamma) = \int_0^1 (f'(t) + g'(t))\,dt = f(1)-f(0) + g(1)-g(0) = 2. 
    \end{align*}
    Therefore, $\gamma$ is a $F$-geodesics. This implies that there exist (infinitely many) $F$-geodesics connecting the points $(0,0)$ and $(1,1)$ that are different from (a reparametrization of) a line segment. 

    On the other hand, if $\rho=(\rho_1,\rho_2) \in \ac([0,1];\R^2)$ is such that $\rho(0) = (0,0)$ and $\rho(1)=(0,1)$. Then 
    \begin{align*}
        \l(\rho) = \int_0^1 |\rho'_1(t)|\,dt + \int_0^1 |\rho'_2(t)|\,dt \geq 1 + \int_0^1 |\dot \rho_2(t)|\,dt \geq 1,
    \end{align*}
    and the two equalities are both equalities if and only if $\rho_1$ is strictly increasing and $\rho_2\equiv 0$. Therefore, up to reparametrization, the line segment $t\mapsto (t,0)$ is the only $F$-geodesic connecting the points $(0,0)$ and $(1,0)$.

    A visual representation of this example is provided in Figure \ref{fig:1}. The thick black and blue lines represent $\graph(F)$ and the boundary of $K_F$ respectively.
    
\end{example}

\section{Proof of Theorem \ref{Cor_Main}}

Let $\Omega\subseteq \Rn$ be an arbitrary set.
A point $x\in \Omega$ is an \emph{extremal point of $\Omega$}, and we write $x\in \Extr(\Omega)$ if $x$ cannot be written as a strictly convex combination of any two other points in $\Omega$. Clearly, if $K\subseteq\Rn$ is a convex subset, then $\Extr(K)\subseteq \partial K$. 

The next result is well-known in the literature, and will play a key role in the sequel.

\begin{lemma}[Krein-Milman theorem]\label{convExtr}
    Every compact convex set in $\Rn$ is the convex hull of its extremal points.
\end{lemma}

\begin{lemma}\label{tangent=extremal}
    Let $F\in \Integrand$ be an integrand. If $\d(F)(v)=1$ (i.e. $v\in \partial O_F$), then $v\in \Ort(\partial K_F)$ if and only if $v\in \Extr(O_F)$.
\end{lemma}

\begin{proof}
    Fix $v\in \Rn$ such that $\d(F)(v)=1$.
    If $v\in \Ort(\partial K_F)$, then exists $\overline{x}\in \partial K_F$ such that $\overline{x} + v^\perp$ is the affine tangent space of $\partial K_F$ at $\overline{x}$. Thus, there exists only one supporting hyperplane of $K_F$ at $\overline{x}$. Suppose that $\Tilde{u},\Tilde{w}\in O_F$ and $\lambda\in (0,1)$ are such that $v= (1-\lambda)\Tilde{u} + \lambda\Tilde{w}$.
    Then
    \begin{align*}
        1 = \d(F)(v) = \langle v, \overline{x}\rangle = (1-\lambda)\langle \Tilde{u},\overline{x}\rangle + \lambda \langle \Tilde{w},\overline{x}\rangle.
    \end{align*}
    This implies
    \begin{align*}
        1 =\d(F)(\Tilde{u}) = \langle \Tilde{u},\overline{x}\rangle\quad \text{and}\quad1 =\d(F)(\Tilde{w})= \langle \Tilde{w},\overline{x}\rangle.
    \end{align*}
    Therefore, by Lemma \ref{interiorConvexHull}\ref{keyLem_**}, $\Tilde{u} + \Tilde{u}^\perp$ and $\Tilde{w} + \Tilde{w}^\perp$ are supporting hyperplanes of $K_F$ at $\overline{x}$. By uniqueness of the tangent space, $\Tilde{u}=\Tilde{w}=v$.

    Viceversa, if $v= (1-\lambda)\Tilde{u} + \lambda \Tilde{w}$ for some $\Tilde{u},\Tilde{w}\in O_F\backslash\{v\}$ and $\lambda\in (0,1)$, then, arguing as before, one proves the existence of three different supporting hyperplanes for $K_F$ at $\overline{x}$. Therefore $v\not \in \Ort(\partial K_F)$.
    
\end{proof}

Let us introduce the following notation. If $v\in \Rn$, we define $\gamma_v\in\ac([0,1];\Rn)$ as $\gamma_v(t):= tv$ for every $0\leq t\leq 1$. The \emph{concatenation of two paths} $\gamma,\rho\in \ac([0,1];\Rn)$ such that $\gamma(0)=\rho(0)=0$ is the path $\gamma \diamond \rho \in \ac([0,1];\Rn)$ defined as
\begin{align*}
    \gamma\diamond \rho(t):=
    \begin{cases}
        \gamma(2t)&, \text{ if }0\leq t\leq \frac{1}{2}\\
        \rho(2t-1) + \gamma(1)&, \text{ if }\frac{1}{2}<t\leq 1
    \end{cases}.
\end{align*}
Observe that if $\gamma,\rho,\sigma \in \ac([0,1];\Rn)$, then the paths $\gamma\diamond (\rho\diamond \sigma)$ and $(\gamma \diamond \rho )\diamond \sigma$ differ only by a reparametrization. Therefore, with a small abuse of notation, when the choice of the parametrization is not important, we may simply write $\gamma_N\diamond \cdots \diamond \gamma_1$ for the concatenation of $N$ paths such that $\gamma_j(0)=0$ for every $1\leq j\leq N$.

Fix a vector $v\in \Rn$ and suppose $v = u_1 + \cdots + u_N$ for some vectors $u_1,...,u_N\in \Rn$. Using the definition of $\ldf(\cdot)$ and the 1-homogeneity of $\d(F)$, one immediately shows
    \begin{align}
        \ldf(\gamma_v) = \d(F)(v) \quad \text{and}\quad \ldf(\gamma_{u_{N}}\diamond \cdots \diamond \gamma_{u_1}) = \sum_{j=1}^N \d(F)({u_j}).
    \end{align}
    Moreover, by Lemma \ref{interiorConvexHull}\ref{keyLem_*}, Jensen's inequality for sums and the 1-homogeneity of $\d(F)$ it follows that
    \begin{align*}
        \d(F)(v) = N\d(F)\left(\frac{u_1+\cdots + u_N}{N}\right) \leq \sum_{j=1}^N \d(F)(u_j).
    \end{align*}
    Therefore,
    \begin{align}\label{ineq}
         \ldf(\gamma_{ u_1 + \cdots + u_N}) \leq \ldf(\gamma_{u_N}\diamond \cdots \diamond \gamma_{u_1}) \quad\forall u_1,..., u_N\in \Rn.
    \end{align}
    On the other hand, if there exist $\Tilde{u}_1,...,\Tilde{u}_N\in \Rn$ and $\lambda_1,...,\lambda_N\in [0,1]$ with $\lambda_1+...+\lambda_N=1$ such that 
    \begin{align*}
        v = \sum_{j=1}^N \lambda_j \Tilde{u}_j \quad \text{and}\quad \d(F)(\Tilde{u}_j)=\d(F)(v)\,\forall 1\leq j\leq N,
    \end{align*}
    then, setting $u_j := \lambda_j \Tilde{u}_j$ for each $1\leq j\leq N$ one obtains equality in \eqref{ineq}.

\begin{proof}[Proof of Theorem \ref{Cor_Main}]
    Let $F\in \Integrand$ be an integrand and $x,y\in \Rn$ be two distinct points. Without loss of generality, suppose $x=0$ and $\d(F)(y)=1$

    \ref{cor_i} If $y\in \Ort(\partial K_F)$, then exists $\overline{x}\in \partial K_F$ such that $y/|y| = N_{\partial K_F}(\overline{x})$ is the (outer) unit normal to $\partial K_F$ at $\overline{x}$. Therefore there exists one unique supporting hyperplane for $K_F$ at $\overline{x}$ and this is $\overline{x} + y^\perp$. Applying Theorem \ref{th_MainResult}, every geodesic $\gamma\in \Geo(0,y)$ must satisfy $\dot\gamma(t)\in \Cont(F)\cap \spn\{y\}$ for almost every $t\in(0,1)$, where
    \begin{align*}
        \spn\{y\}:= \{\lambda y: \lambda \geq 0\}.
    \end{align*}
    Since, by Lemma \ref{interiorConvexHull}\ref{keyLem_iv} $\Ort(\partial K_F)\subseteq \Cont(F)$, then $\Cont(F)\cap \spn\{y\}=\spn\{y\}$. In particular, $\gamma$ must be a reparametrization of $\gamma_y$.

    \ref{cor_ii} Suppose now $y\not \in \Ort(\partial K_F)$. Then, by Lemma \ref{tangent=extremal}, $y$ is not extremal in $O_F$. Using Lemma \ref{convExtr} and  Lemma \ref{tangent=extremal} once again, we find $\lambda\in (0,1)$ and $\Tilde{u},\Tilde{w}\in \Ort(\partial K_F)$ with $\d(F)(\Tilde{u})=\d(F)(\Tilde{w})=1$ and  such that $y = u+w$, where  $u:=(1-\lambda)\Tilde{u}$ and $w:=\lambda \Tilde{w}$. For any $\tau \in [0,1]$, consider the curve $\sigma^\tau:[0,1]\to\Rn$ defined as
    \begin{align*}
        \sigma^\tau(t):= \gamma_{(1-\tau)u}\diamond\gamma_{w}\diamond \gamma_{\tau u}.
    \end{align*}
    Then, for any $\tau_1\neq \tau_2$ we have that $\sigma^{\tau_1}$, is not a reparametrization of $\sigma^{\tau_2}$ and, by Lemma \ref{interiorConvexHull}\ref{keyLem_iv}, Corollary \ref{norm=D} and the above remark,
    \begin{align}\label{proof_cor3}
        \l(\sigma^\tau) =  \d(F)(u) + \d(F)(w) = \d(F)(y) = \norm{y}_F\quad \forall\tau \in(0,1).
    \end{align}
    As $\sigma^\tau(0)=0$ and $\sigma^\tau(1)=y$ for every $\tau\in[0,1]$, Theorem \ref{th_MainResult} and \eqref{proof_cor3} prove that $\{\sigma^\tau:0\leq\tau\leq 1\}$ is an infinite family of $F$-geodesic, each one of them identifying a different element in ${\Geo}_{/\sim}(0,y)$.
    
\end{proof}

\end{document}